\documentclass[12pt,compress]{article}
\usepackage{mathrsfs}
\usepackage{amsfonts}

\usepackage{amsmath}
\usepackage{amssymb}
\usepackage{amsthm}
\usepackage{enumerate}

\usepackage{longtable,booktabs,makecell,multirow,tabularx}
\usepackage{graphicx}

\usepackage{geometry}
\newtheorem{theorem}{Theorem}[section]

\newtheorem{lemma}[theorem]{Lemma}

\numberwithin{equation}{section}

\def\UU{{\mathcal U}}

\usepackage{latexsym}
\usepackage{amssymb}

\usepackage{color}
\usepackage{hyperref}

\usepackage{cleveref}
\usepackage[numbers,sort&compress]{natbib}

\newtheorem{innercustomgeneric}{\customgenericname}
\providecommand{\customgenericname}{}
\newcommand{\newcustomtheorem}[2]{%
	\newenvironment{#1}[1]
	{%
		\renewcommand\customgenericname{#2}%
		\renewcommand\theinnercustomgeneric{##1}%
		\innercustomgeneric
	}
	{\endinnercustomgeneric}
}

\newcustomtheorem{customtheorem}{Theorem}
\newcustomtheorem{customlemma}{Lemma}
\newcustomtheorem{customdefinition}{Definition}
\newcustomtheorem{customhypothesis}{Hypothesis}
\newcustomtheorem{customcorollary}{Corollary}
\newcustomtheorem{customproposition}{Proposition}
\newcustomtheorem{customremark}{Remark}
\newcustomtheorem{customquestion}{Question}
\newcustomtheorem{customproblem}{Problem}
\newcustomtheorem{customexample}{Example}
\newcustomtheorem{customConjecture}{Conjecture}
\newcustomtheorem{customconclusion}{Conclusion}

\theoremstyle{definition}
\newtheorem{definition}[theorem]{Definition}

\usepackage{indentfirst}
\begin{document}
\begin{sloppypar}



\begin{center}
{\Large \bf Finite groups with some subgroups of prime power order satisfying the partial $ \Pi $-property

\renewcommand{\thefootnote}{\fnsymbol{footnote}}

\footnotetext[1]
{Corresponding author.}

}\end{center}

                        \vskip0.6cm
\begin{center}

                       Zhengtian Qiu$ ^{1} $ and Adolfo Ballester-Bolinches$^{2, \ast}$

                            \vskip0.5cm

       $ ^{1} $School of Mathematics and Statistics, Southwest University,

       Beibei, Chongqing 400715, P. R. China

       $ ^{2} $Departament de Matem$ \grave{\mathrm{a}} $tiques, Universitat de Val$ \grave{\mathrm{e}} $ncia, Dr. Moliner, 50,
       46100 Burjassot, Valencia, Spain

      E-mail addresses:  qztqzt506@163.com \,    \  Adolfo.Ballester@uv.es

\end{center}

                          \vskip0.5cm

\begin{abstract}
	Let $ H $ be a subgroup of a finite group $ G $.  We say that $ H $ satisfies the partial $ \Pi  $-property in $ G $ if there exists a $G$-chief series $ \varGamma_{G}: 1 =G_{0} < G_{1} < \cdot\cdot\cdot < G_{n}= G $ of $ G $ such that $ | G / G_{i-1} : N _{G/G_{i-1}} (HG_{i-1}/G_{i-1}\cap G_{i}/G_{i-1})| $ is a $ \pi (HG_{i-1}/G_{i-1}\cap G_{i}/G_{i-1}) $-number for every $ G $-chief factor $ G_{i}/G_{i-1} $ of $ \varGamma_{G} $, $1\leq i\leq n$.  In this paper, we investigate the structure of a finite group $ G $  under the assumption that some subgroups of prime power order satisfy  the partial $ \Pi  $-property.

\end{abstract}

{\hspace{0.88cm} \small \textbf{Keywords:} Finite group, $ p $-soluble group, partial $ \Pi  $-property.}
	
\vskip0.1in

{\hspace{0.88cm} \small \textbf{Mathematics Subject Classification (2020):} 20D10, 20D20.}

\section{Introduction}


All groups considered in this paper are finite.


The present paper is a further contribution to the research project that analyses the impact of the embedding of some relevant families of subgroups on the structure of a group. The subgroup embedding property studied here is the \textit{partial  $ \Pi $-property} introduced by Chen and Guo in \cite{Chen-2013}, which  generalises a large number of known embedding properties (see \cite[Section 7]{Chen-2013}), and has raised a considerable interest in recent years.

\begin{definition}
We say that a subgroup $H$ of a group $G$ satisfies the \textit{partial  $ \Pi $-property} in $ G $ if there exists a $G$-chief series $ \varGamma_{G}: 1 =G_{0} < G_{1} < \cdot\cdot\cdot < G_{n}= G $ of $ G $ such that $ | G / G_{i-1} : N _{G/G_{i-1}} (HG_{i-1}/G_{i-1}\cap G_{i}/G_{i-1})| $ is a $ \pi (HG_{i-1}/G_{i-1}\cap G_{i}/G_{i-1}) $-number for every $ G $-chief factor $ G_{i}/G_{i-1} $ of $ \varGamma_{G} $, $1\leq i\leq n$.
\end{definition}

As usual, $ \pi(G) $ denotes the set of all primes dividing the order of a group $G$.

Our aim here is to solve three problems concerning the structural influence of the partial $\Pi$-property of some subgroups of prime power order and confirm that the partial $\Pi$-property is a true generalisation of the partial {\rm CAP}-property studied in \cite{Adolfo-JA, Adolfo-JPAA, Adolfo-2015, Qian-2020}.

Chen and Guo proved in \cite[Proposition 1.6]{Chen-2013} that a group $G$ in which every subgroup of a Sylow $ p $-subgroup $ P $ of order $p$, $p$ a prime, and every cyclic subgroup of order $4$ (if $p =2$ and $P$ is not quaternion-free), satisfy the partial $\Pi$-property  is $p$-supersoluble, that is, every chief factor of $G$ divisible by $p$ is cyclic.

Our first main result provides a structural information of a group $G$ in which every subgroup of a Sylow $ p $-subgroup of order $p^2$ satisfies the partial $\Pi$-property. Without loss of generality, we may assume that the Sylow $p$-subgroups of $G$ have order at least $p^2$ and $  O_{p'}(G) = 1 $.

\begin{customtheorem}{A}\label{2-minimal}
	Let $p$ be a prime, and let $ P $ be a Sylow $ p $-subgroup of a group $ G $. Assume that $  O_{p'}(G) = 1 $ and $|P| \geq p^2$. If every subgroup of $ P $ of order $p^2$ satisfies the partial $ \Pi $-property in $ G $, then  $ P=O_{p}(G)$ and $ G = P\rtimes H $, the semidirect product of $P$ and $H$, where $ H \in {\rm Hall}_{p'}(G) $, and one of the following statements holds:
	
	
	{\rm (1)} $ G $ is $ p $-supersoluble;
	
	{\rm (2)} $P$ is a minimal normal subgroup of $ G $ of order $p^2$;
	
	{\rm (3)} $ |P|\geq p^{4} $ and $H$ is cyclic. Furthermore, $ P = V_{1} \times \cdot\cdot\cdot \times V_{s} $, $ s \geq 2 $, where $ V_{i} $ is an $ H $-isomorphic irreducible $ \mathbb{F}_{p}[H] $-module of dimension $2$ for all $ 1\leq i\leq s $, where  $\mathbb{F}_{p}$ is the finite field of $p$-elements.
\end{customtheorem}

Chen and Guo proved in \cite[Proposition 1.4]{Chen-2013} that if the maximal subgroups of a Sylow $p$-subgroup of a group $G$ satisfy the $\Pi$-property in $G$, then either $G$ is $p$-supersoluble or the order of the Sylow $p$-subgroups of $G$ is $p$.

Our second main result describes the structure of a group $G$ in which the second maximal subgroups of a Sylow $p$-subgroup satisfy the partial $ \Pi $-property in $G$. We may assume that the Sylow $p$-subgroups of $G$ have order at least $p^2$ and $  O_{p'}(G) = 1 $.

\begin{customtheorem}{B}\label{2-maximal}
	Let $ P $ be a Sylow $ p $-subgroup of a group $ G $. Assume that $ O_{p'}(G) = 1 $ and $|P| \geq p^2$. If every $2$-maximal subgroup of $ P $ satisfies the partial $ \Pi $-property in $ G $, then one of the following  statements holds:
	
	
	{\rm (1)} $ G $ is $ p $-supersoluble;
	
	{\rm (2)}  $ |P|=p^{2} $ and $P=O_{p}(G)$ is a minimal normal subgroup of $G$;
	
	{\rm (3)} $ |P|=p^{2} $ and $ G $ is non-$ p $-soluble;
	
	{\rm (4)}  $ p = 2 $ and  $ P $ is isomorphic to $ Q_{8} $;

	{\rm (5)} $ |P|\geq p^{3} $ and $ G = P\rtimes H $, where $ H \in {\rm Hall}_{p'}(G) $ is cyclic. Furthermore, $ \Phi(P) $ is the intersection of all $ 2 $-maximal subgroups of $ P $, and $ P/\Phi(P) = V_{1} \times \cdot\cdot\cdot \times V_{s} $, where $ V_{i} $ is an $ H $-isomorphic irreducible $ \mathbb{F}_{p} [H] $-module of dimension $ 2 $ for all $1\leq i\leq s$.
\end{customtheorem}

A remarkable extension of Chen and Guo's results is due to the first author, Liu and Chen.

\begin{theorem}[{\cite[Theorem 1.3]{Qiu}}]\label{know}
	Let $ P $ be a Sylow $ p $-subgroup of a group $ G $, and let $ d $ be a power of $ p $ such that $ 1 < d < |P| $. Assume that every subgroup of  $ P $ of order $ d $, and every cyclic subgroup of $ P $ of order $ 4 $ (when $ d=2 $ and $ P $ is not quaternion-free) satisfy the partial $ \Pi $-property in $ G $.  Then $ G $ is $ p $-soluble of $p$-length at most $1$.
\end{theorem}

Our third main result describes the structure of the groups of $p$-rank greater than $1$ satisfying the hypotheses of Theorem~\ref{know}. Recall that the $p$-rank of a $p$-soluble group $G$ is the largest integer $ k $ such that $ G $ has a chief factor of order $ p^{k} $. Note that a $p$-soluble group $G$ is $p$-supersoluble if and only if the $p$-rank of $G$ is equal to $1$.




  \begin{customtheorem}{C}\label{prime-power}
Let $ P \in {\rm Syl}_{p}(G) $ and let $ d $ be a  power of $ p $ such that $ 1 < d < |P| $. Assume that every subgroup of $ P $ of order $ d $ satisfies the partial $ \Pi $-property in $ G $, and assume further that every cyclic subgroup of $ P $ of order $ 4 $  satisfies the partial $ \Pi $-property in $ G $ when $ d = 2 $ and $ P $  is not quaternion-free. If $ O_{p'}(G)=1 $ and the $p$-rank of $G$ is greater than $1$, then $ G=P\rtimes H $, where $ H\in \mathrm{Hall}_{p'}(G) $.  Furthermore, if $ V $ is an irreducible $ \mathbb{F}_{p}[H] $-submodule of $ P/\Phi(P) $,  and write
  	
  	\begin{center}
  		$ \mathrm{dim}_{\mathbb{F}_{p}}V=k $, $ n=\log_p d  -\log_p |\Phi(P)| $, and $ m=\log_p |P/\Phi(P)| $,
  	\end{center}
  	
  	\noindent	then the following statements hold:
  	
  	{\rm (1)} $ P/\Phi(P) $ is a homogeneous $ \mathbb{F}_{p}[H] $-module and every irreducible submodule of $ P/\Phi(P) $ is not absolutely irreducible;

  	{\rm (2)} $ n \geq k\geq 2 $, $ k$ divides $\gcd(m, n) $;
  	
  	{\rm (3)} $ H $ is cyclic.
  \end{customtheorem}

\section{Preliminaries}

Before we can begin our proofs of Theorems~\ref{2-minimal}, ~\ref{2-maximal} and~\ref{prime-power}, we need a number of preliminary results. The first of these is very useful in induction arguments.

\begin{lemma}[{\cite[Lemma 2.1]{Chen-2013}}]\label{over}
	Let $ H $ be a subgroup of a group $ G $ and $ N\unlhd G $. If either $ N \leq H $ or $ \gcd(|H|, |N|)=1 $ and $ H $ satisfies the partial $ \Pi $-property
	in $ G $, then $ HN/N $ satisfies the partial $ \Pi $-property in $ G/N $.
\end{lemma}


The next lemma is crucial to prove useful results on the partial $\Pi$-property.

\begin{lemma}[{\cite[Lemma 2.3]{Qiu-Liu-Chen}}]\label{pass}
	Let $ H $ be a $ p $-subgroup of a group $ G $  and $ N $ be  a normal subgroup of $ G $  containing $ H $. If  $ H $ satisfies the partial $ \Pi $-property in $ G $, then $ G $ has  a chief series $$ \varOmega_{G}: 1 =G^{*}_{0} < G^{*}_{1} < \cdot\cdot\cdot <G^{*}_{r}=N < \cdot\cdot\cdot < G^{*}_{n}= G $$  passing through $ N $ such that $ |G:N_{G}(HG^{*}_{i-1}\cap G^{*}_{i})| $ is a $ p $-number  for every $ G $-chief factor $ G^{*}_{i}/G^{*}_{i-1} $ $ (1\leq i\leq n) $ of $ \varOmega_{G} $.
\end{lemma}

The order of the minimal normal subgroups of groups with subgroups of prime power order satisfying the partial $\Pi$-property is quite constrained.

\begin{lemma}[{\cite[Lemma 2.1]{Qiu}}]\label{order-d}
	Let $ P $ be a Sylow $ p $-subgroup of a group $ G $ and let $ d $ be a power of $ p $ such that $ 1 < d < |P| $. Assume that every  subgroup of  $ P $ of order $ d $ satisfies the partial $ \Pi $-property in $ G $. Then:
	
	{\rm (1)} Every minimal normal subgroup of $ G $ is either a $ p' $-group or a $ p $-group of order at most $ d $.
	
	{\rm (2)} If $ G $  has  a minimal normal subgroup  of order $ d $, then  every minimal normal $ p $-subgroup of $ G $ has order $ d $.
\end{lemma}

Recall that the \textit{supersoluble hypercenter} $ Z_{\UU}(G) $ of a group $ G $ is the product of all normal subgroups $ H $ of $ G $, such that all $ G $-chief factors under $ H $ have prime order, and the \textit{$ p $-supersoluble hypercenter} $ Z_{\UU_{p}}(G) $ is the product of all normal subgroups $ H $ of $ G $, such that all $ p $-$ G $-chief factors under $ H $ have order $ p $ for some fixed prime $ p $. Note that a normal $ p $-subgroup is contained in $ Z_{\UU}(G) $ if and only if it is contained in $ Z_{\UU_{p}}(G) $.

\begin{lemma}[{\cite[Lemma 2.11]{Qiu}}]\label{in}
	Let $ P $ be a normal $ p $-subgroup of a group $ G $. If all cyclic subgroups of $ P $ of order $ p $ or $ 4 $ (when $ P $ is not quaternion-free) are contained in $ Z_{\UU}(G) $, then $ P \leq Z_{\UU}(G) $.
\end{lemma}

If $ P $ is either an odd order $ p $-group or a quaternion-free $ 2 $-group, then we use $ \Omega(P) $ to denote the subgroup $ \Omega_{1} (P) $.  Otherwise, $ \Omega (P) = \Omega_{2} (P) $.

Our next lemma is a consequence of \cite[Lemma 2.8]{Chen-xiaoyu-2016}.

\begin{lemma}\label{hypercenter}
	Let $ P $ be a normal $ p $-subgroup of a group $ G $ and $ D $ a Thompson critical subgroup of $ P $ (see \cite[page 186]{Gorenstein-1980}). If $ P/\Phi(P) \leq Z_{\UU}(G/\Phi(P)) $ or  $ \Omega(D) \leq Z_{\UU}(G) $, then $ P \leq  Z_{\UU}(G) $.
\end{lemma}


\begin{lemma}[{\cite[Lemma 2.10]{Chen-xiaoyu-2016}}]\label{critical}
	Let $ D $ be a Thompson critical subgroup of a non-trivial $ p $-group $ P $.
	
	
	{\rm (1)} If $ p > 2 $, then the exponent of $ \Omega_{1}(D) $ is $ p $.
	
	{\rm (2)} If $ P $ is an abelian $ 2 $-group, then the exponent of $ \Omega_{1}(D) $ is $ 2 $.
	
	{\rm (3)} If $ p = 2 $, then the exponent of $ \Omega_{2}(D) $ is at most $ 4 $.
\end{lemma}

\begin{lemma}[{\cite[Lemma 3.1]{Ward}}]\label{charcteristic}
	Let $ P $ be a non-abelian quaternion-free $ 2 $-group. Then $ P $ has a characteristic subgroup of index $ 2 $.
\end{lemma}

\begin{lemma}[{\cite[Lemma 2.10]{Su-2014}}]\label{phi}
	Let $ p $ be a prime, $ E $ a normal subgroup of a group $ G $ such that $ p $ divides the order of $ E $. Then $ E \leq  Z_{\UU_{p}}(G) $
	if and only if $ E /\Phi(E) \leq Z_{\UU_{p}}(G/\Phi(E)) $.
\end{lemma}

\begin{lemma}\label{also}
	Let $ p $ be a prime. Assume that $ N $ is a minimal normal subgroup of a group $ G $ and $ K $ a subgroup of $ G $. Assume that $ |N|=|K|=p $. If $ NK $ satisfies the partial $ \Pi $-property in $ G $, then $ K $ also satisfies the partial $ \Pi $-property in $ G $.
\end{lemma}

\begin{proof}
	If $ K = N $, there is nothing to prove. Hence we can assume that $ NK $ has order $ p^{2} $. There exists a chief series $ \varGamma_{G}: 1 =G_{0} < G_{1} < \cdot\cdot\cdot < G_{n}= G $ of $ G $ such that $ |G/G_{i-1} : N _{G/G_{i-1}}(NKG_{i-1}/G_{i-1}\cap G_{i}/G_{i-1})| $ is a $ p $-number, for every $ G $-chief factor $ G_{i}/G_{i-1}$ of $ \varGamma_{G} $, $1\leq i\leq n$ . Our goal is to show that if $ G_{i}/G_{i-1}$ is a chief factor of $ \varGamma_{G} $, it follows that $ |G/G_{i-1} : N _{G/G_{i-1}}(KG_{i-1}/G_{i-1}\cap G_{i}/G_{i-1})| $ is a $ p $-number. Obviously we may assume that $ |( KG_{i-1}/G_{i-1})\cap (G_{i}/G_{i-1})| \neq 1 $.

Assume that $ |( NKG_{i-1}/G_{i-1}) \cap (G_{i}/G_{i-1})| = p^{2} $, then $ NK\cap G_{i-1} = 1 $ and $ NKG_{i-1}\leq G_{i} $. Note that  $ G_{i}= NG_{i-1} $. Hence $|(NKG_{i-1}/G_{i-1}) \cap (G_{i}/G_{i-1})| = |(NKG_{i-1}/G_{i-1}) \cap (NG_{i-1}/G_{i-1})| = |NG_{i-1}/G_{i-1}| = p $, contrary to our
	assumption. Hence $ |(NKG_{i-1}/G_{i-1})\cap (G_{i}/G_{i-1})|\leq p $ and so $ |(KG_{i-1}/G_{i-1})\cap (G_{i}/G_{i-1})| = p = |(NKG_{i-1}/G_{i-1})\cap (G_{i}/G_{i-1})| $. Thus $ (KG_{i-1}/G_{i-1})\cap (G_{i}/G_{i-1}) =  (NKG_{i-1}/G_{i-1})\cap (G_{i}/G_{i-1}) $ and therefore $ | G / G_{i-1} : N _{G/G_{i-1}} (KG_{i-1}/G_{i-1}\cap G_{i}/G_{i-1})| $ is a $ p $-number. Consequently, $ K $ satisfies the partial $ \Pi $-property in $ G $.
\end{proof}

\begin{lemma}\label{Normal}
	Let  $ P\in {\rm Syl}_{p}(G) $, $ d $ be a power of $ p $ with $ p\leq d\leq |P| $, and let $ N $ be a minimal normal subgroup of $ G $ whose order is divisible by $d$. Assume  that every subgroup of $ P $ of order $ d $ satisfies the partial $ \Pi $-property in $ G $. Then $ N $ is an elementary abelian subgroup of $G$ of order $ d $.
\end{lemma}

\begin{proof}
	Let $ H $ be a normal subgroup of $ P $ of order $ d $ such that $ H \leq N $. By hypothesis, $ H $ satisfies the partial $ \Pi $-property in $ G $. It follows from  Lemma \ref{pass} that $ |G : N_{G}(H)| $ is a $ p $-number, and so $ H\unlhd G $. The minimality of $ N $ yields that $ H=N $. Thus $ N $ is an elementary abelian subgroup of $G$ of order $ d $.
\end{proof}

\begin{lemma}\label{two}
	Let $ P $ be a normal Sylow $ p $-subgroup of a group $ G $. Suppose that a $ 2 $-maximal subgroup $ H $ of $ P $ satisfies the partial $ \Pi $-property in $ G $.  Then $ H $ is a partial {\rm CAP}-subgroup of $ G $.
\end{lemma}

\begin{proof}
	Since $ H\leq P\unlhd G $, it follows from Lemma \ref{pass} that $ G $ has  a chief series $$ \varOmega_{G}: 1 =G^{*}_{0} < G^{*}_{1} < \cdots <G^{*}_{r}=P < \cdots < G^{*}_{n}= G $$  passing through $ P $ such that $ |G:N_{G}(HG^{*}_{i-1}\cap G^{*}_{i})| $ is a $ p $-number  for every $ G $-chief factor $ G^{*}_{i}/G^{*}_{i-1} $  of $ \varOmega_{G} $, $ 1\leq i\leq n $.  Observe that $ Z(P)\cap G^{*}_{1} \neq 1 $. Since $ Z(P)$ is normal in $G$ and $G^{*}_{1}$ is a minimal normal subgroup of $G$, it follows that $G^{*}_{1}  \leq Z(P)$. In particular, $H \cap G^{*}_{1} $ is normalised by $P$. Since $ |G:N_{G}(H\cap G^{*}_{1})| $ is a $ p $-number, it follows that $H \cap G^{*}_{1} $ is normal in $G$. Therefore either $H \cap G^{*}_{1}  = 1$ or $H \cap G^{*}_{1}  = G^{*}_{1}$.
	
	
	Assume that $H \cap G^{*}_{1}  = 1$. For every $ G $-chief factor $ G_{i}/G_{i-1} $ with $ 1<i\leq r $, we see that either $P = HG^{*}_{i-1}$ or $HG^{*}_{i-1}$ is a maximal subgroup of $P$. In both cases, $P$ normalises $HG^{*}_{i-1}$ and hence $P \leq N_{G}(HG^{*}_{i-1}\cap G^{*}_{i})$. Since  $ |G:N_{G}(HG^{*}_{i-1}\cap G^{*}_{i})| $ is a $ p $-number, it follows that $HG^{*}_{i-1}\cap G^{*}_{i}$ is a normal subgroup of $G$. Therefore $H$ covers or avoids $ G^{*}_{i}/G^{*}_{i-1} $ with $ 1<i\leq r $. Thus $H$ covers or avoids every $ G $-chief factor of $ \varOmega_{G} $.  This implies that $H$ is a partial {\rm CAP}-subgroup of $G$, as desired. Hence we can assume that $H \cap G^{*}_{1}  = G^{*}_{1}$. By Lemma \ref{over}, the hypotheses are inherited by
	$ G/G^{*}_{1} $. By induction, $ H/G^{*}_{1} $ is is a partial {\rm CAP}-subgroup of $ G/G^{*}_{1} $. Hence $ G/G^{*}_{1} $ has  a chief series $$ \varGamma_{G/G^{*}_{1}}: 1 =G_{1}/G^{*}_{1} <G_{2}/G^{*}_{1}< \cdots < G_{n}/G^{*}_{1}= G/G^{*}_{1} $$   such that $ H/G^{*}_{1} $ covers or avoids every $ G/G^{*}_{1} $-chief factor $ G_{j}/G^{*}_{1}\big /G_{j-1}/G^{*}_{1} $ of $ \varGamma_{G/G^{*}_{1}} $, $ 2\leq j\leq n $. Therefore, $ G_{j}/G^{*}_{1}\leq HG_{j-1}/G^{*}_{1} $ or $ (H\cap G_{j})/G^{*}_{1}\leq G_{j-1}/G^{*}_{1} $. Consequently, $ G_{j}\leq HG_{j-1} $ or $ H\cap G_{j}\leq G_{j-1} $ with $ 2\leq j\leq n $. Hence $ H $  covers or avoids
	every $ G $-chief factor of a chief series of $ G $. Thus $ H $ is a partial {\rm CAP}-subgroup of $ G $, as desired.
\end{proof}


Recall that a subgroup $ H $ of a group $ G $ is said to be \textit{complemented} in  $ G $ if there exists a subgroup $ K $ of $ G $ such that $ G = HK $ and $ H \cap K = 1 $. In this case, $ K $ is called a complement of $ H $ in $ G $.


\begin{lemma}\label{completed}
	Let $ G $ be a  group with an elementary abelian normal Sylow $ p $-subgroup $ P $. Suppose that $ H $ is a subgroup of $ P $. Then the following statements are equivalent:
	
	{\rm (1)} $ H $ satisfies the partial $ \Pi  $-property in $ G $.
	
	{\rm (2)} $ H $ is complemented in $ G $.
\end{lemma}

\begin{proof}
	
	(1)$ \Rightarrow $(2). We argue by induction on $|G|$. If $ H=1 $ there is nothing to prove. By the Schur-Zassenhaus Theorem (\cite[Theorem~A.11.3]{MR1169099}), $P$ is complemented in $G$ by a Hall $p'$-subgroup $U$ of $G$. Hence we may assume that $1<H<P $. By Lemma \ref{pass}, $ G $ has  a chief series $$ \varOmega_{G}: 1 =G^{*}_{0} < G^{*}_{1} < \cdot\cdot\cdot <G^{*}_{r}=P < \cdot\cdot\cdot < G^{*}_{n}= G $$  passing through $ P $ such that $ |G:N_{G}(HG^{*}_{i-1}\cap G^{*}_{i})| $ is a $ p $-number  for every $ G $-chief factor $ G^{*}_{i}/G^{*}_{i-1} $  of $ \varOmega_{G} $, $1\leq i\leq n$. In particular, $ |G:N_{G}(H\cap G^{*}_{1})| $ is a $ p $-number, and so $ G=N_{G}(H\cap G^{*}_{1})P $. Since $ P $ is elementary abelian, we have  $ (H\cap G_{1}^{*})^{G}=H^{P}\cap G_{1}^{*}=H\cap G_{1}^{*} $. The minimality of $ G_{1}^{*} $ implies that $ H\cap G_{1}^{*}=G_{1}^{*} $ or $ 1 $. If $ H\cap G_{1}^{*}=G_{1}^{*} $, then $ G_{1}^{*}\leq H $. By Lemma \ref{over}, $ H/G_{1}^{*} $ satisfies the partial $ \Pi $-property in $ G/G_{1}^{*} $. By induction,  $ H/G_{1}^{*} $ has a complement in $ G/G_{1}^{*} $, $ K/G_{1}^{*} $ say. Thus $ G = HK $ and $ H\cap K=G_{1}^{*} $. Hence $ P = H(P \cap  K) $ and $ P \cap K \unlhd G $. By Maschke's theorem (\cite[Theorem~A.11.4]{MR1169099}), there exists a normal subgroup $ L $ of $ G $ such that $ P \cap K = G_{1}^{*} \times L $.  Let $ U $ be a Hall $ p' $-subgroup of $ G $. In this case, $LU$ is a complement of $H$ in $G$, as required.
	
	If $ H\cap G_{1}^{*}=1 $, then $ HG_{1}^{*}/G_{1}^{*} $ satisfies the partial $ \Pi  $-property in $ G/G_{1}^{*} $.  Hence $ HG_{1}^{*}/G_{1}^{*} $ has a complement, $ T/G_{1}^{*} $ say, in $ G/G_{1}^{*} $ by induction. It follows that $ G = HG_{1}^{*}T = HT $ and $ HG_{1}^{*}\cap T =G_{1}^{*} $. Since $ H \cap G_{1}^{*} = 1 $, we have  $ H \cap T= 1 $. Hence $ T $ is a complement of $ H $ in $ G $, as required.

	
	
	(2)$ \Rightarrow $(1). Suppose that $ H $ is complemented in $ G $. We show that $H$ satisfies the partial $\Pi$-property in $G$ by induction on $|G|$. It is no loss to assume that $1<H<P$. Let $ B $ be a complement of $ H $ in $ G $. Then $ P\cap B\unlhd G $ and $ P = H \times (P\cap B) $. Since $H(P \cap B)/P\cap B$ is complemented in $G/P \cap B$, it follows that $H(P \cap B)/(P\cap B)$ satisfies the partial
	$\Pi$-property in $G/(P \cap B)$. Now, if $C/D$ is a $G$-chief factor below $P \cap B$, then $H\cap C\leq D$. Consequently, $H$ satisfies the partial $ \Pi  $-property in $ G $.
\end{proof}

The following lemma is a significant result proved by Zeng.

\begin{lemma}[{\cite[Theorem B]{Zeng}}]\label{Completed}
	Assume that a  $ p' $-group $ H $ acts faithfully on an elementary abelian $ p $-group $ P $. Assume that  $ G=P\rtimes H $ and let $ d $ be a power of $ p $ such that $ 1<d<|P| $. Then all subgroups of $ P $ of order $ d $ are complemented in $ G $ if and only if one of the following statements holds:
	
	{\rm (1)} $ G $ is supersoluble.
	
	{\rm (2)} $ H $ is cyclic and $ P $ is a homogeneous $ \mathbb{F}_{p}[H] $-module with all its irreducible $ \mathbb{F}_{p}[H] $-submodules having
	dimension $ k>1 $. Furthermore, $ k$ divides $\gcd(\log_p d, \log_p |P|) $.
\end{lemma}

\begin{lemma}[{\cite[Lemma 3.3(4)]{Qian-2020}}]\label{cyclic}
	Let $ H $ be a $ p' $-group and let $ V $ be a faithful and irreducible $ \mathbb{F}_{p}[H] $-module. Assume that $ \mathrm{dim}_{\mathbb{F}_{p}}V $ is prime. Then $ H $ is cyclic if and only if $ V $ is not absolutely irreducible.
\end{lemma}

\section{Proofs}

We work towards the proofs of our main results.

\begin{lemma}\label{ele}
	Let $ P \in {\rm Syl}_{p}(G) $ and let $ d $ be a  power of $ p $ such that $ p^{2} \leq d < |P| $. Assume that $ O_{p'}(G)=1 $ and every subgroup of $ P $ of order $ d $ satisfies the partial $ \Pi $-property in $ G $. If $ G $ has a minimal normal subgroup of order $ d $, then $G/\mathrm{Soc}(G)$ is cyclic, $ \Phi(G) = 1 $ and $ P = \mathrm{Soc}(G)$ is a homogeneous $ \mathbb{F}_{p}[G] $-module.
	
	
	
\end{lemma}

\begin{proof}
	By Theorem \ref{know}, $ G $ is $ p $-soluble with $ p $-length at most $ 1 $. Since $ O_{p'}(G)=1 $, it follows that $P$ is normal in $G$ and so the Schur-Zassenhaus Theorem implies that $ G=P\rtimes H $, where $ H\in \mathrm{Hall}_{p'}(G) $. Since $ G $ has a minimal normal subgroup of order $ d $, it follows from Lemma \ref{order-d}(2) that every minimal normal subgroup of $ G $ is  elementary abelian  of order $ d $. Thus $ N = \mathrm{Soc}(G)\leq P = O_p(G)$ and $N$ is a completely reducible $G$-module by Maschke's theorem.
	
First we prove that $G/N$ is $p$-supersoluble. Assume not and let $ T \unlhd G $ be of minimal order such that $  T \leq P $ and $ T/(T\cap N) \nleq  Z_{\UU_{p}}(G/(T\cap N)) $. Then $ T\cap N= V_{1} \times \cdots \times V_{s}$, $ s\geq  1 $, where  $ V_{i} $ is a minimal normal subgroup of $G$ for all $1\leq i\leq s$.


	Assume that $ T $ has two different maximal $ G $-invariant subgroups, $ U_{1} $, $ U_{2} $ say. By minimality of $ T $, $ U_{j}/(U_{j}\cap N)\leq Z_{\UU_{p}}(G/(U_{j}\cap N)) $  for each $ j \in \{ 1, 2 \} $. This implies that $ U_{1}N/N $ and $ U_{2}N/N $  are contained in $ Z_{\UU_{p}}(G/N) $. Consequently, $ T/(T \cap N)\cong TN/N = (U_{1}N/N)(U_{2}N/N) \leq Z_{\UU_{p}}(G/N) $, a contradiction. Hence $ T $ has  a unique maximal $ G $-invariant subgroup, $ U $ say. Clearly,  $ T > T \cap N $. Hence $ U\geq T \cap N $. Since $ U/(U \cap N)\leq Z_{\UU_{p}}(G/(U\cap N)) $  and $ U \cap N = T \cap N $, we deduce that $ U/(T \cap N)\leq Z_{\UU_{p}}(G/(T\cap N)) $. Hence   $ T/U\nleq Z_{\UU_{p}}(G/U) $.

	
	
	
	
	
	
	Write $ N_{0}=T\cap N $ and denote with bars the images in $ \overline{G}=G/N_{0} $. Then $ \overline{U}=\overline{T}\cap Z_{\UU_{p}}(\overline{G}) $. We claim that the exponent of $ \overline{T} $ is $ p $ or $ 4 $ (when $ \overline{T} $  is not quaternion-free).
	
	Let $ \overline{D} $ be a Thompson critical subgroup of $ \overline{T} $.  If $ \Omega(\overline{T})< \overline{T} $, then $ \Omega(\overline{T}) \leq Z_{\UU_{p}}(\overline{G}) $.
	By Lemma \ref{hypercenter}, we have that $ \overline{T}\leq Z_{\UU_{p}}(\overline{G}) $, a  contradiction. Thus $ \overline{T} = \overline{D} = \Omega(\overline{D}) $.  If $ \overline{T} $ is a non-abelian quaternion-free $ 2 $-group, then $ \overline{T} $ has a characteristic subgroup $ \overline{T_{0}} $ of index $ 2 $ by Lemma \ref{charcteristic}. The uniqueness of $ U $ implies $ T_{0}=U $. Thus  $ \overline{T} \leq Z_{\UU_{p}}(\overline{G}) $, which is impossible. Hence if $ \overline{T} $ is a non-abelian $ 2 $-group, then $ \overline{T} $ is not quaternion-free.  By Lemma \ref{critical}, the exponent of $ \overline{T} $ is $ p $ or $ 4 $ (when $ \overline{T} $  is not quaternion-free), as claimed.
	
	There exists  a normal subgroup $ \overline{Q} $ of $ \overline{P} $  such that $ \overline{U}< \overline{Q}< \overline{T} $ and $ |\overline{Q}:\overline{U}|=p $.  Let $ \overline{x}\in \overline{Q}\setminus \overline{U} $. Then $ \langle \overline{x} \rangle \overline{U}=\overline{Q} $ and  $ \langle \overline{x} \rangle $ has order $ p $ or $ 4 $. Since $N_{0} $ is  elementary abelian, we have  $ o(x)\leq p^{2} $ or $ o(x)\leq 8 $.
	
	If $ o(x) > d $, then $ o(x) = 8 $, $ p=2 $ and $ d = 4 $.  Since $ x^{2}\in U $ and $ o(x^{2}) = 4 $, we have $ \Phi(U) > 1 $.  Let $ A $ be a minimal normal subgroup of $ G $ contained in $  \Phi(U) $. Then $ A\leq T\cap N=N_{0} $.  Without loss of generality, we may assume $ A=V_{1} $. Set  $ \widetilde{G} = G/(V_{2} \times \cdots \times V_{s}) $ and denote with $\widetilde{X}$ the image of a subgroup $X$ of $G$ in $ \widetilde{G}$. Clearly, $ \widetilde{N_{0}} \leq \widetilde{\Phi(U)} \leq \Phi(\widetilde{U}) $ and $ \widetilde{U}/\widetilde{N_{0}}\leq Z_{\UU_{p}}(\widetilde{G}/\widetilde{N_{0}}) $. By Lemma \ref{phi}, we have $ \widetilde{U}\leq Z_{\UU_{p}}(\widetilde{G}) $. This implies that $ V_{1}\cong \widetilde{N_{0}}$ is cyclic, a contradiction.

	Suppose  that $ o(x) \leq d $.
	Let $ L $ be a  subroup of $ Q $ of  order  $ d $ such that $ \langle x \rangle \leq L $. Then $ Q=LU $. By hypothesis, $ L $ satisfies the partial $ \Pi  $-property in $ G $. Applying Lemma \ref{pass}, $ G $ has  a chief series
	$$ \varOmega_{G}: 1 =G^{*}_{0} < G^{*}_{1} < \cdot\cdot\cdot <G^{*}_{r-1} <G^{*}_{r}=T < \cdot\cdot\cdot < G^{*}_{n}= G $$
	passing through $ T $ such that $ |G:N_{G}(LG^{*}_{i-1}\cap G^{*}_{i})| $ is a $ p $-number  for every $ G $-chief factor $ G^{*}_{i}/G^{*}_{i-1}$ of $ \varOmega_{G} $, $ 1\leq i\leq n $. Since $ U $ is the unique maximal  $G$-invariant subgroup of $T$, it follows that $ G_{r-1}^{*}=U $. Therefore $ |G:N_{G}(LU\cap T|=|G:N_{G}(Q)| $ is a $ p $-number. It follows that $ Q\unlhd G $, which contradicts the fact that $ U $ is the  unique maximal $ G $-invariant subgroup of  $ T $.
	
	Consequently, $G/N$ is $p$-supersoluble. Assume that $ \Phi(G) > 1 $. Let $ B $ be a minimal normal subgroup of $G$ contained in $ \Phi(G) $. Then $ B\leq N=\mathrm{Soc}(G) $. Since $N$ is a completely reducible $G$-module, there exists $ K\unlhd G $ such
	that $ N = B \times K $.  Observe that $ N/K\leq \Phi(G)K/K \leq \Phi(G/K) $ and $ (G/K)/(N/K)\cong G/N $ is $ p $-supersoluble. It follows by Lemma \ref{phi} that $ G/K $ is $ p $-supersoluble, and thus $ B\cong N/K$ is cyclic, a contradiction. As a consequence, $ \Phi(G) = 1 $. Then $ \mathrm{Soc}(G)=N=P $. By Lemma~\ref{completed}, every subgroup of $P$ of order $ d $ is complemented in $G$. Applying Lemma~\ref{Completed}, we conclude that $G/\mathrm{Soc}(G)$ is cyclic and $P$ is a homogeneous $ \mathbb{F}_{p}[G] $-module.
\end{proof}

\begin{lemma}\label{orderp2}
	Let $ P \in {\rm Syl}_{p}(G) $ and let $ d $ be a  power of $ p $ such that $ 1 < d < |P| $. Assume that every subgroup of $ P $ of order $ d $ satisfies the partial $ \Pi $-property in $ G $, and assume further that every cyclic subgroup of $ P $ of order $ 4 $  satisfies the partial $ \Pi $-property in $ G $ when $ d = 2 $ and $ P $  is not quaternion-free. If the $p$-rank of $G$ is greater than $1$, then  $ d\geq p^{2}|\Phi(P)| $.
\end{lemma}

\begin{proof}
	Assume the result is not true and let $ G $ be a counterexample of minimal order. Then $ d<p^{2}|\Phi(P)| $. We may assume that $ O_{p'}(G)=1 $. By Theorem \ref{know}, $G$ is $p$-soluble of $p$-length at most $1$. Hence $P$ is normal in $G$ and $ \mathrm{Soc}(G)\leq P $.  If $ d=p $, then $ G $ is $ p $-supersoluble by \cite[Proposition 1.6]{Chen-2013}, a contradiction. Hence $ d \geq p^{2} $ and $ \Phi(P)>1 $. By Lemma \ref{order-d}(1), every minimal normal subgroup of $ G $ is a $ p $-group of order at most $ d $. If $ G $  has  a minimal normal subgroup  of order $ d $, then $ \Phi(P)=\Phi(G)=1 $ by Lemma \ref{ele}, a contradiction. Hence every minimal normal subgroup of $ G $ has order at most $ \frac{d}{p} $.
	
	Let $ N $ be a minimal normal subgroup of $G$ contained in $ \Phi(P) $. Clearly, $ \frac{d}{|N|}\geq p $ since $ |N|\leq \frac{d}{p} $. Note that the class of all $ p $-supersoluble groups is a saturated formation. As $ N \leq \Phi(P) \leq \Phi(G) $ and the $p$-rank of $G$ is greater than one, we can  conclude that the $p$-rank of $G/N$ is greater than $1$. By Lemma \ref{over}, every subgroup of $ P/N $ of order $ \frac{d}{|N|} $ satisfies the partial $ \Pi $-property in $ G/N $. If $p\geq 3 $, then $ G/N $ satisfies the hypotheses of the lemma. The minimal choice of $ G $ implies that $ \frac{d}{|N|}\geq p^{2}|\Phi(P/N)| = p^{2}|\Phi(P)/N| $, and hence $ d\geq p^{2}|\Phi(P)| $, a contradiction. This yields that $ p=2 $ and $ |N|=\frac{d}{2} $.
	
	Therefore  every subgroup of $ P/N $ of order $ 2 $ satisfies the partial $ \Pi $-property in $ G/N $. By \cite[Proposition 1.6]{Chen-2013}, we may assume that $ P/N $ is not quaternion-free. By Lemma \ref{Normal} that every minimal normal subgroup of $ G/N $ of order divisible by $ 2 $ has order $ 2 $. Thus every minimal normal subgroup of $G$ different from $ N $ has order $ 2 $.
	
	If $ G/N $ is $2$-supersoluble, then $ G $ is $2$-supersoluble, a contradiction. Hence $ G/N $ is not $2$-supersoluble. Let $ M$ be a normal subgroup of $G$ of minimal order containing $N$ such that $ M/N \nleq Z_{\UU}(G/N) $ and $ M\leq P $. Let $U/N$ be a minimal normal subgroup of $G/N$ contained in $M/N$. Then $ U/N = M/N \cap Z_{\UU}(G/N) $.  By Lemma \ref{in}, $ M/N $ possesses a cyclic subgroup $ \langle x\rangle N/N $ of order $ 2 $ or $ 4 $  such that $ x\in M \setminus U $.
	
	Denote with bars the images in $ \overline{G}=G/N$. We will show that the exponent of $ \overline{M} $ is $ 2 $ or $ 4 $ (when $ \overline{M} $  is not quaternion-free).   Let $ \overline{D} $ be a Thompson critical subgroup of $ \overline{M} $.  If $ \Omega(\overline{D})< \overline{M} $, then $ \Omega(\overline{D})\leq \overline{U} $ by the  minimal choice of $M$. Thus  $ \Omega(\overline{D})\leq Z_{\UU}(\overline{G}) $.
	By Lemma \ref{hypercenter}, we have that $ \overline{M}\leq Z_{\UU}(\overline{G}) $, a  contradiction. Hence  $ \overline{M} = \overline{D} = \Omega(\overline{D}) $.  If $ \overline{M} $ is a non-abelian quaternion-free $ 2 $-group, then $ \overline{M} $ has a characteristic subgroup $ \overline{M_{0}} $ of index $ 2 $ by Lemma \ref{charcteristic}. The minimal choice of $ M $ implies $ M_{0}=U $. Thus  $ \overline{M} \leq Z_{\UU}(\overline{G}) $, which is impossible. Hence if $\overline{M} $ is a non-abelian $ 2 $-group, then $ \overline{M} $ is not quaternion-free.  By Lemma \ref{critical}, the exponent of $ \overline{M} $ is $ 2 $ or $ 4 $ (when $ \overline{M} $  is not quaternion-free), as desired.
	
	
	Suppose that $ d\geq 8 $ and  $ N $ is the unique minimal normal subgroup of $ G $. There exists  a normal subgroup $ \overline{Q} $ of $ \overline{P} $  such that $ \overline{U}< \overline{Q}< \overline{M} $ and $ |\overline{Q}:\overline{U}|=2 $.  Let $ \overline{y}\in \overline{Q}\setminus \overline{U} $. Then $ |\langle \overline{y} \rangle|\leq 4 $ and $ \langle \overline{y} \rangle \overline{U}=\overline{Q} $. Since $ N $ is  elementary abelian, we have  $ o(y)\leq 8 $.  Let $ L $ be a  subroup of $ Q $ of  order  $ d $ such that $ \langle y \rangle \leq L $. Then $ Q=LU $. By hypothesis, $ L $ satisfies the partial $ \Pi  $-property in $ G $. By the uniqueness of $ N $ and Lemma \ref{pass}, $ G $ has  a chief series
	$$ \varOmega_{G}: 1 =G^{*}_{0} < G^{*}_{1}=N < \cdot\cdot\cdot <G^{*}_{r-1} <G^{*}_{r}=M < \cdot\cdot\cdot < G^{*}_{n}= G $$
	passing through $ M $ and $ N $ such that $ |G:N_{G}(LG^{*}_{i-1}\cap G^{*}_{i})| $ is a $ p $-number  for every $ G $-chief factor $ G^{*}_{i}/G^{*}_{i-1}$  of $ \varOmega_{G} $,  $1\leq i\leq n$. The choice of $M$  yields $ G_{r-1}^{*}=U $. Therefore $ |G:N_{G}(LU\cap M)|=|G:N_{G}(Q)| $ is a $2$-number. It follows that $ Q\unlhd G $, which contradicts the fact that $ U $ is the  unique maximal $ G $-invariant subgroup of $ M $.
	
	Suppose that $ d\geq 8 $ and  $ N $ is not the unique minimal normal subgroup of $ G $. Then $ G $ has a minimal normal subgroup of order $ 2 $, $ T $ say.  Then the $2$-rank of $G/T$ is greater than $1$. Note that every subgroup of $ P/T $ of order $ \frac{d}{2} $ satisfies the partial $ \Pi  $-property in $ G/T $ and $ \frac{d}{2}\geq 4 $. The minimal choice of $ G $ implies $ \frac{d}{2} \geq 2^{2}|\Phi(P/T)| $. Obviously $ |\Phi(P/T)| \geq \frac{|\Phi(P)|}{2} $, and it follows that $ d\geq 2^{2}|\Phi(P)| $, a contradiction.
	
	Suppose that $ d = 4 $. Then $ |N|=2 $, and thus $ N\leq Z(G) $. By hypothesis,  every subgroup of $ P $ of order $ 4 $ satisfies the partial $ \Pi $-property in $ G $. Now we claim that  every subgroup $ X $ of $ P $ of order
	$ 2 $ satisfies the partial $ \Pi $-property in $ G $. We may assume that $ N\not =X $. Then $ XN $ has order $ 4 $. Hence $ XN $ satisfies the partial $ \Pi $-property in $ G $. By Lemma  \ref{also}, $ X $ satisfies the partial $ \Pi $-property in $ G $. By  \cite[Proposition 1.6]{Chen-2013}$,  G $ is $ 2 $-supersoluble, a contradiction. The proof of the lemma is now complete.
\end{proof}

\begin{lemma}\label{dim}
Let $ P \in {\rm Syl}_{p}(G) $ and let $ d $ be a  power of $ p $ such that $ 1 < d < |P| $. Assume that every subgroup of $ P $ of order $ d $ satisfies the partial $ \Pi $-property in $ G $, and assume further that every cyclic subgroup of $ P $ of order $ 4 $  satisfies the partial $ \Pi $-property in $ G $ when $ d = 2 $ and $ P $  is not quaternion-free. Assume that $ G=P\rtimes H $, where $ H\in \mathrm{Hall}_{p'}(G) $, $P$ is elementary abelian and the $p$-rank of $G$ is greater than $1$. Then the following statements hold:

	{\rm (1)} Then $ P $ is a homogeneous $ \mathbb{F}_{p}[H] $-module with all its irreducible $ \mathbb{F}_{p}[H] $-submodules having
	dimension $ k>1 $ and $ k $ divides $ \gcd(\log_pd, \log_p |P|) $.
	
	{\rm (2)} Every irreducible $ \mathbb{F}_{p}[H] $-submodule of $ P $ is not absolutely irreducible.
\end{lemma}

\begin{proof}
 It is clear that $P$ is a completely reducible $ \mathbb{F}_{p}[H] $-module by Maschke's theorem. Write $ P = V_{1} \times \cdots \times V_{t} $, where $ V_{i} $ is an irreducible $ \mathbb{F}_{p}[H] $-submodule of $ P $ for all $ 1\leq i\leq t $. Write $ \mathrm{dim}V_{i}= n_{i} $ and suppose that $ n_{1} \leq n_{2} \leq \cdots \leq n_{t}  $.  By Lemma \ref{order-d}(1), we have $ d\geq p^{n_{t}} $. Since $G$ is not $p$-supersoluble, it turns out that $ d\geq p^{n_{t}}\geq p^{2} $.

(1) We argue by induction on $|G|$. If $ d=p^{n_{t}} $, then $ n_{1}=n_{2}=\cdots=n_{t} $ by Lemma \ref{order-d}(2). In this case, the conclusion follows by Lemmas~\ref{completed} and~\ref{Completed}. We may assume that $ d>p^{n_{t}} $ and $ t>1 $.
	
	Suppose that $ t = 2 $. By Lemma \ref{over}, every subgroup of $ P/V_{1} $ of order $ \frac{d}{|V_{1}|} $ satisfies the partial $ \Pi $-property in $ G/N $. Applying Lemma \ref{order-d}(1), we see that $ |P/V_{1}|=p^{n_{2}}\leq \frac{d}{p^{n_{1}}} $. Consequently, $ |P|=p^{n_{1}+n_{2}}\leq d<|P| $, a contradiction. Therefore $ t\geq 3 $. Note that every subgroup of $ P/V_{j} $ of order $ \frac{d}{|V_{j}|} $ satisfies the partial $ \Pi $-property in $ G/V_{j} $ and the $p$-rank of $ G/V_{j}$ is equal to $n_{t}>1 $ for $ j=1, 2 $. Applying the inductive hypothesis to $ G/V_{1} $ and $ G/V_{2} $, we deduce that $ V_{1}, ... , V_{t} $ are isomorphic $ \mathbb{F}_{p}[H] $-modules of the same dimension $ k>1 $. Applying the inductive hypothesis to $ G/V_{1} $, we have $ k \big| \gcd(\log_p d-k, \log_p |P|-k) $. Hence $ k\big|\gcd(\log_pd, \log_p |P|) $, as wanted.
	
	(2) Let $ U $ and $ V $ be two isomorphic irreducible $ \mathbb{F}_{p}[H] $-modules of $ P $, and let $ \varphi : U \mapsto V $ be an $H$-isomorphism. Suppose that $ U $ is absolutely irreducible. Then $ \mathrm{End}_{\mathbb{F}_{p}[H]}(U) \cong\mathbb{F}_{p}  $. According to \cite[Proposition~B.8.2]{MR1169099}, the minimal normal subgroups of $G$ contained in $ U\times V $ are exactly
	\begin{equation*}
		U \; \text{and} \; V_{m}=\{u^{m}\varphi(u)|u\in U \}, \; \text{where} \; m = 0, 1, ..., p-1.
	\end{equation*}
	
	
	\noindent Let $ 1 \not = u_{1} \in U  $, and let $ A $ be a subgroup of $ U\times V $ of order $ d $ such that $ \langle u_{1}, \varphi(u_{1})\rangle \subseteq A $. It is rather easy to see that $ \mathrm{Core}_{G}(A)=1 $. By Lemma \ref{pass}, there exists a minimal normal subgroup $ B $ of $G$ contained in $ U\times V $ such that $ |G:N_{G}(A\cap B)| $ is a $ p $-number. Since  $ P $ is elementary abelian, we have $ A\cap B\unlhd G $. The minimality of $ B $ yields that $ A\cap B=B $ or $ 1 $. If $ A\cap B=B $, then $ \mathrm{Core}_{G}(A)>1 $, a contradiction. Therefore $ A\cap B=1 $. However,  $ 1\not =u_{1}\in A\cap U $ and $ 1\not=u_{1}^{m}\varphi(u_{1})\in A \cap V_{m} $ for $ m = 0, 1, ..., p-1 $. This  contradiction yields that $ U $ is not absolutely irreducible.
\end{proof}

\begin{lemma}\label{contained-in}
	Let $ G $ be a $ p $-soluble group of $p$-rank greater than $1$ and let $P$ be a Sylow $p$-subgroup of $ G $ with $ |P|\geq p^{2} $. Suppose that every $ 2 $-maximal subgroup of $ P $ satisfies the partial $ \Pi $-property in $ G $. If $ Q $ is a $ 2 $-maximal subgroup of $ P $, then  $ \Phi(P) $ is contained in $ Q $.
\end{lemma}

\begin{proof}
	We argue by induction on $ |G| $. By Lemma \ref{over}, $G/O_{p'}(G)$ inherits the hypotheses of the lemma. Hence we may assume that $ O_{p'}(G) = 1 $. In view of Theorem \ref{know}, we have $ P=O_{p}(G)=F(G) $. By Lemma \ref{two}, every $ 2 $-maximal subgroup of $ P $ is  a partial {\rm CAP}-subgroup of $ G $. Applying \cite[Lemmas 10 and 12]{Adolfo-JPAA}, we conclude that $ \Phi(P) $ is contained in $ Q $.
\end{proof}

\begin{proof}[Proof of Theorem \ref{2-minimal}]
	
Applying \cite[Theorem 1.3]{Qiu-Liu-Chen}, we have that $ G $ is $ p $-soluble of $p$-length at most $1$. Since $ O_{p'}(G)=1 $, it follows that $P=O_{p}(G)$ and so $ G=P\rtimes H $, where $ H\in \mathrm{Hall}_{p'}(G) $, by the Schur-Zassenhaus Theorem.
	
Assume that $ G $ is not $ p $-supersoluble. If $ |P|=p^{2} $, then $ P$ is a minimal normal subgroup of $ G $. Consequently, $ G $ is of type~(2).
	
Suppose that $ |P|\geq p^{3} $. By Lemma \ref{orderp2}, we have $ \Phi(P)=1 $.
	By Lemma \ref{dim}(1), $ P = V_{1} \times \cdot\cdot\cdot \times V_{s} $, $ s \geq 2 $, where $ V_{i} $ is an $ H $-isomorphic irreducible $ \mathbb{F}_{p}[H] $-module of dimension $ 2 $ for all $1\leq i\leq s$.  Furthermore,  $ s \geq 2 $ and $ |P|\geq p^{4} $. Applying Lemma \ref{cyclic} and and Lemma \ref{dim}(2), we have that $ H $ is cyclic. Therefore, $ G $ is of type (3).
\end{proof}

\begin{proof}[Proof of Theorem \ref{2-maximal}]
	Assume that $ |P| \geq p^{2} $ and $ G $ is not $ p $-supersoluble. Suppose that $ |P|=p^{2} $. If $ G $ is $ p $-soluble, then $P=O_{p}(G)$ is a minimal normal subgroup of $G$ and $G$ is of type~(2). If $ G $ is non-$ p $-soluble, then $ G $ is of type (3).

Assume that $ |P| \geq p^{3} $. If $ \frac{|P|}{p^{2}}=2 $, then $ |P|=8 $. Since $G$ is not
$2$-supersoluble, we can apply \cite[Proposition 1.6]{Chen-2013}  to conclude that $ P\cong Q_{8} $. Hence $ G $ is of type~(4).
	
Suppose that $ d = \frac{|P|}{p^{2}}>2 $. By Theorem \ref{know}, we have that $ G $ is $ p $-soluble of $p$-length at most $1$. Since $ O_{p'}(G)=1 $, it follows that $P=O_{p}(G)$ and so $ G=P\rtimes H $, where $ H\in \mathrm{Hall}_{p'}(G) $, by the Schur-Zassenhaus Theorem.  By Lemma \ref{orderp2}, we have $ \frac{|P|}{|\Phi(P)|}> \frac{d}{|\Phi(P)|}\geq p^{2} $. Applying Lemma \ref{over}, every $ 2 $-maximal subgroup of $ P/\Phi(P) $ satisfies the partial $ \Pi $-property in $ G/\Phi(P) $.
	By Lemma \ref{dim}, we have  $ P/\Phi(P) = V_{1} \times \cdot\cdot\cdot \times V_{s} $, where $ V_{i} $ is an $ H $-isomorphic irreducible $ \mathbb{F}_{p} [H] $-module of dimension $ 2 $, which is not absolutely irreducible, for all $1\leq i\leq s $.  It follows from \ref{cyclic} that $ H $ is cyclic. Note that $ |P|  > d \geq p^{2}|\Phi(P)| $ and $ \Phi(P) $ is contained in every $ 2 $-maximal subgroup of $ P $ by Lemma \ref{contained-in}. Hence $ \Phi(P) $ equals to the intersection of all $ 2 $-maximal subgroups of $ P $. Therefore $ G $ is of type (5).
\end{proof}

\begin{proof}[Proof of Theorem \ref{prime-power}]
	Applying Theorem~\ref{know}, we have that $ G $ is $ p $-soluble of $p$-length at most $1$. Since $ O_{p'}(G)=1 $, it follows that $P=O_{p}(G)$ and so $ G=P\rtimes H $, where $ H\in \mathrm{Hall}_{p'}(G) $, by the Schur-Zassenhaus Theorem.
By Lemma \ref{orderp2}, we know  $ \frac{d}{|\Phi(P)|}\geq p^{2} $. According to Lemma \ref{over}, every subgroup of $ P/\Phi(P) $ of order $ \frac{d}{|\Phi(P)|} $ satisfies  the partial $ \Pi $-property in $ G/\Phi(P) $. Since $ P/\Phi(P) $ is  elementary abelian, all subgroups of $ P/\Phi(P) $ of order $ \frac{d}{|\Phi(P)|} $ are complemented in $ G/\Phi(P) $ by Lemma \ref{completed}. By  Lemma~\ref{dim}, $G$ satisfies  Statements~(1) and~(2). We also have $O_{p'}(G/\Phi(P))=1$ and so $H\Phi(P)/\Phi(P)$ acts faithfully on $ P/\Phi(P) $ by \cite[Theorem~6.3.2]{Gorenstein-1980}. By Lemma \ref{Completed}, $H$ is cyclic.
\end{proof}


\section*{Acknowledgments}

    The first author has been supported by the Natural Science Foundation Project of CQ (No.cstc2021jcyj-msxmX0426). The first author also thanks the China Scholarship Council and the Departament de Matem$ \grave{\mathrm{a}} $tiques of the Universitat de Val$ \grave{\mathrm{e}} $ncia for its hospitality.


    \small


\end{sloppypar}
\end{document}